\documentclass[12pt]{article}

\usepackage{color}
\def \red {\textcolor{red} }

\usepackage{epsfig}
\usepackage{cite}
\usepackage{latexsym,amsmath}
\usepackage{amsfonts}
\usepackage{float}
\usepackage{pict2e}
\usepackage{tikz}
\usepackage{caption}
\usepackage{amsfonts}
\usepackage{amssymb}
\usepackage{mathrsfs}
\usepackage{amsmath}
\usepackage{enumerate}
\usepackage{graphicx}
\usepackage{subfig}
\usepackage{MnSymbol}
\usepackage{mathtools}
\begin{document}

	\bibliographystyle{plain}
		
		\pagestyle{myheadings}
		\thispagestyle{empty}
		\newtheorem{theorem}{Theorem}
		\newtheorem{corollary}[theorem]{Corollary}
		\newtheorem{definition}{Definition}
		\newtheorem{guess}{Conjecture}
		\newtheorem{problem}{Problem}
			\newtheorem{claim}[theorem]{Claim}
		\newtheorem{question}{Question}
		\newtheorem{lemma}[theorem]{Lemma}
		\newtheorem{proposition}[theorem]{Proposition}
		\newtheorem{observation}[theorem]{Observation}
		\newenvironment{proof}{\noindent {\bf
				Proof.}}{\hfill\rule{3mm}{3mm}\par\medskip}
		\newcommand{\remark}{\medskip\par\noindent {\bf Remark.~~}}
		\newcommand{\pp}{{\it p.}}
		\newcommand{\de}{\em}
		\newtheorem{example}{Example}

\title{\bf Minimum non-chromatic-$\lambda$-choosable graphs}

\author{Jialu Zhu\thanks{Department of Mathematics, Zhejiang Normal University, Email: jialuzhu@zjnu.edu.cn, }         \and
	Xuding Zhu\thanks{Department of Mathematics, Zhejiang Normal University, Email: xdzhu@zjnu.edu.cn, Grant Numbers: NSFC 11971438, U20A2068. ZJNSFC   LD19A010001.}}

\maketitle


\begin{abstract}

  For a multi-set $\lambda=\{k_1,k_2, \ldots, k_q\}$ of positive integers,
  let $k_{\lambda} = \sum_{i=1}^q k_i$.
  A $\lambda$-list assignment of $G$ is a list assignment $L$ of $G$ such that the colour set $\bigcup_{v \in V(G)}L(v)$ can be partitioned into the disjoint union $C_1 \cup C_2 \cup \ldots \cup C_q$ of $q$ sets so that for each $i$ and each vertex $v$ of $G$, $|L(v) \cap C_i| \ge k_i$. We say $G$ is $\lambda$-choosable if $G$ is $L$-colourable for any $\lambda$-list assignment $L$ of $G$. The concept of $\lambda$-choosability puts $k$-colourability and $k$-choosability in the same framework: If $\lambda = \{k\}$, then $\lambda$-choosability is equivalent to $k$-choosability; if $\lambda$   consists of $k $ copies of $1$, then   $\lambda$-choosability is equivalent to $k $-colourability.
If $G$ is $\lambda$-choosable, then $G$ is $k_{\lambda}$-colourable.
On the other hand, there are $k_{\lambda}$-colourable graphs that are not $\lambda$-choosable, provided that  $\lambda$   contains an integer larger than $1$. Let $\phi(\lambda)$ be the minimum number of vertices in a  $k_{\lambda}$-colourable non-$\lambda$-choosable graph. This paper determines the value of $\phi(\lambda)$ for all $\lambda$.

\medskip

\noindent  {\bf Keywords:}   $\lambda$-list assignment, chromatic-choosable graph, non-chromatic-$\lambda$-choosable graph.
\end{abstract}

\section{Introduction}
 A \emph{proper colouring} of a graph $G$ is a mapping $f:V(G) \rightarrow \mathbb{N}$ such that $f(u) \neq f(v)$ for any edge $uv$ of $E(G)$. 
The \emph{chromatic number} $\chi(G)$ of   $G$ is the minimum positive integer $k$ such that $G$ is $k$-colourable, i.e., there is a proper colouring $f$ of $G$ using colours  from $  \{1,2,\ldots,k\}$.
 The \emph{choice number} $ch(G)$ of $G$ is the minimum positive integer $k$ such that $G$ is $k$-choosable, i.e., if $L$ is a list assignment which assigns to each vertex $v$ a set $L(v) \subseteq \mathbb{N}$ of at least $k$ integers as permissible colours, there is a proper colouring $f$ of $G$ such that $f(v) \in L(v)$ for each vertex $v$. 

It follows from the definitions that $\chi(G) \le ch(G)$ for any graph $G$, and it was shown in \cite{ERT1980} that bipartite graphs can have arbitrarily large choice number. An interesting problem is for which graphs   $G$, $\chi(G)=ch(G)$. Such graphs are called {\em chromatic-choosable}.
Chromatic-choosable graphs  {have} been studied extensively in the literature. There are a few challenging conjectures that assert certain families of graphs are chromatic-choosable. The most famous problem concerning this concept is perhaps the {\em list colouring conjecture}, which asserts that line graphs are chromatic-choosable \cite{BH1985}. Another problem concerning chromatic-choosable graphs that has attracted a lot of attention is the minimum order of a non-chromatic-choosable graph with given chromatic number. 
For a positive integer $k$, let
 $$\phi(k)=\min\{n: \text{ there exists a non-$k$-choosable $k$-chromatic $n$-vertex graph}\}.$$ 
Ohba \cite{Ohba2002} conjectured that $\phi(k) \ge 2k+2$. In other words, $k$-colourable graphs on at most $2k+1$ vertices are $k$-choosable.
This conjecture  was studied in many papers \cite{KMZ2014, NWWZ2015,NRW2015,Ohba2002,Ohba2004,RS2005,KSW2011,SHZL2009,SHZWZ2008}, and was  finally confirmed by  Noel, Reed and Wu  \cite{NRW2015}. This upper \red{lower} bound is tight if $k$ is even, i.e., $\phi(k)=2k+2$ when $k$ is even.   Noel \cite{Noel} further conjectured that if $k$ is odd, then $k$-colourable graphs on at most $2k+2$ vertices are also $k$-choosable.   Recently, the authors of this paper confirmed Noel's conjecture \cite{zhuzhu}, and   determined the value of $\phi(k)$ for all $k$. 
\begin{theorem}  
	\label{thm-phik}
	For   $k \ge 2$, \[
	\phi(k) = \begin{cases} 2k+2, &\text{ if $k$ is even},\cr 
	2k+3, &\text{ if $k$ is odd}.\cr
	\end{cases}
	\]
\end{theorem}

The concept of {$\lambda$-choosability} is a refinement of choosability introduced in \cite{zhu2019}.
Assume that $\lambda=\{k_1,k_2, \ldots, k_q\}$ is a multi-set of positive integers. Let $k_{\lambda} = \sum_{i=1}^q k_i$ and $|\lambda| = q$.  
A {\em $\lambda$-list assignment} of $G$ is a list assignment $L$ such that the colour set $\bigcup_{v \in V(G)}L(v)$ can be partitioned into the disjoint union $C_1 \cup C_2 \cup \ldots \cup C_q$ of $q$ sets so that for each $i$ and each vertex $v$ of $G$, $|L(v) \cap C_i| \ge k_i$. Note that for each vertex $v$, $|L(v)| \ge \sum_{i=1}^qk_i = k_{\lambda}$. So a $\lambda$-list assignment $L$ is a $k_{\lambda}$-list assignment with some restrictions on the set of possible lists.
We say $G$ is \emph{$\lambda$-choosable} if $G$ is $L$-colourable for any $\lambda$-list assignment $L$ of $G$.

For a positive integer $a$, let $m_{\lambda}(a)$ be the multiplicity of $a$ in $\lambda$. If $m_{\lambda}(a)=m$, then instead of writing $m$ times the integer $a$, we may write $a \star m$. For example, $\lambda=\{1 \star k_1, 2\star k_2,  3\}$ means that $\lambda$ is a  multi-set consisting of $k_1$ copies of $1$, $k_2$ copies of $2$ and one copy of $3$.
 If   $\lambda=\{k  \}$, then $\lambda$-choosability is the same as $k $-choosability; if  
$\lambda = \{1 \star k\}$, then
$\lambda$-choosability is equivalent to $k $-colourability. 
So the concept of $\lambda$-choosability puts $k$-choosability and $k$-colourability in the same framework.

Assume that $\lambda = \{k_1,k_2,\ldots, k_q\}$ and $\lambda'= \{k'_1,k'_2,\ldots,k'_p\}$. We say $ \lambda'$ is a {\em refinement} of $\lambda$ if $p \ge q$ and there is a partition $I_1 \cup I_2 \cup \ldots \cup I_q$ of $\{1,2,\ldots, p\}$ such that $\sum_{j \in I_t}k'_j = k_t$ for $t=1,2,\ldots, q$. We say $\lambda'$ is obtained from $\lambda$ by increasing some parts if $p=q$ and $k_t \le k'_t$ for $t=1,2,\ldots, q$. We
write $\lambda \le \lambda'$ if 
$\lambda'$ is  a  refinement of $\lambda''$, and $\lambda''$ is obtained from $\lambda$ by increasing some parts. 
It follows from the definitions that if $\lambda \le \lambda'$, then every $\lambda$-choosable graph is $\lambda'$-choosable. Conversely, it was proved in \cite{zhu2019} that if $\lambda \not\le \lambda'$, then there is a $\lambda$-choosable graph which is not $\lambda'$-choosable. 
In particular,  $\lambda$-choosability implies $k_{\lambda}$-colourability, and if $\lambda \ne \{1 \star k_{\lambda}\}$, then
there are  $k_{\lambda}$-colourable graphs that are not $\lambda$-choosable.

All the partitions $\lambda$ of a  positive integer $k$ are sandwiched between $\{k\}$ and $\{1 \star k\}$   in the above order. As observed above, $\{k \}$-choosability is the same as $k $-choosability, and $\{1 \star k\}$-choosability is equivalent to $k $-colourability. For other partitions $\lambda$ of $k $,  $\lambda$-choosability reveals a complex hierarchy of colourability of graphs sandwiched between $k $-colourability and $k $-choosability. 

The framework of $\lambda$-choosability provides room to explore strengthenings of   colourability and choosability results.  For example, it was proved by Kermnitz and Voigt that there are planar graphs that are not $\{1,1,2\}$-choosable \cite{KV2018}. This result strengthens  Voigt's result that there are non-4-choosable planar graphs, and shows that the Four Colour Theorem is sharp in the sense of $\lambda$-choosability. This result is further generalized in \cite{YYX}, where it is shown that for $n \ge 5$, there are $K_n$-minor free graphs that are not $\{1 \star (n-3),2\}$-choosable. This result shows the sharpness of  Hadwiger's conjecture in the framework of $\lambda$-choosability. As another example,
Kim and Park \cite{KP} showed that for any positive integer $k$, there are graphs $G$ such that $G^k$ is not chromatic-choosable. 
This result refutes    a conjecture of Woodall and Kostochka \cite{WK}.
A combination of a result in \cite{KP} and a result in \cite{zhuzhulambda} shows that there is a graph $G$, $G^k$ is not $\{1 \star (\chi(G^k)-2),2\}$-choosable. 

The framework of $\lambda$-choosability also provides room to explore
weakenings of difficult conjectures.  For example, it is known that   for any graph $G$, if $k_{\lambda} = \chi'(G)$ and each integer in $\lambda$ is at most 2, then 
$L(G)$ is $\lambda$-choosable \cite{zhu2019}. On the other hand, the following question remains open:  \emph{Is it true that for any graph $G$, the line graph $L(G)$ of $G$ is $\{1 \star (\chi'(G)-3), 3\}$-choosable? }

In this paper, we are interested in minimum non-chromatic-$\lambda$-choosable graphs. Similar to the definition of $\phi(k)$, for a  multi-set $\lambda$ of positive integers, we define $\phi(\lambda)$ as follows:

\begin{definition}
	Assume $\lambda $ is a multi-set of positive integers.
	Let $$\phi(\lambda)=\min\{n: \text{ there exists a non-$\lambda$-choosable $k_{\lambda}$-chromatic $n$-vertex graph}\}.$$
\end{definition}

If $\lambda=\{k\}$, then $\phi(\lambda)=\phi(k)$. We extend Theorem \ref{thm-phik} and  determine the value of $\phi(\lambda)$ for all $\lambda$. 

If $\lambda'$ is a refinement of $\lambda$, then  $\lambda \le \lambda'$ and $k_{\lambda} = k_{\lambda'}$. As every $\lambda$-choosable graph is $\lambda'$-choosable, we have  $\phi(\lambda') \ge \phi(\lambda)$. In particular,
	as any multi-set $\lambda$ is a refinement of $\{k_{\lambda}\}$, we conclude that
	\[
	\phi(\lambda) \ge \phi(k_{\lambda}) = \begin{cases} 2k_{\lambda} +2, &\text{ if $k_{\lambda}$ is even}, \cr
	 2k_{\lambda} +3, &\text{ if $k_{\lambda}$ is odd}.
	 \end{cases} 
	\]

If $\lambda = \{1 \star k\}$, then $\lambda$-choosable is equivalent to $k$-colourable.
In this case, we set $\phi(\lambda) = \infty$. We call such a multi-set $\lambda$ trivial. In the following,
we only consider non-trivial multi-sets of positive integers.

For a non-trivial multi-sets $\lambda$ of positive integers, it is natural that $\phi(\lambda)$ also highly depends on the number of $1$'s in $\lambda$.
Let $m_{\lambda}({\rm odd})$   be the number of integers in $\lambda$ that are odd.  
The following result was proved in  \cite{zhuzhulambda}.
\begin{theorem}
	\label{thm-previous}
	For any non-trivial multi-set $\lambda$ of positive integers, $$2k_{\lambda}+m_{\lambda}(1)+2 \leqslant \phi(\lambda ) \leqslant \min \{2k_{\lambda}+m_{\lambda}({\rm odd})+2, 2k_{\lambda}+5   m_{\lambda}(1)+3\}.$$
\end{theorem}

In this paper, we determine the value of $\phi(\lambda)$ for all  $\lambda$.

\begin{theorem}
	\label{thm-main}
Assume  $\lambda$  is a non-trivial multi-set of positive integers. Then 

$$\phi(\lambda)=\min\{2k_{\lambda}+m_{\lambda}({\rm odd})+2, 2k_{\lambda}+3m_{\lambda}(1)+3\}.$$
\end{theorem}
 
If $m_{\lambda}({\rm odd}) = m_{\lambda}(1)$, then it follows from 
Theorem \ref{thm-previous} that $\phi(\lambda)=2k+m_{\lambda}(1)+2$. 
In the following, we assume $m_{\lambda}({\rm odd}) > m_{\lambda}(1)$.

\section{The  $m_{\lambda}(1) =0$ case} 

This section proves that if $m_{\lambda}(1)=0$ and $m_{\lambda}({\rm odd}) > 0$, then $\phi(\lambda)=2k_{\lambda} +3$. Assume that $k_{\lambda}=k$.
By Theorem \ref{thm-previous},  $$2k+2\le \phi(\lambda)\le 2k+3.$$ We need to show that $\phi(\lambda) \ne 2k+2$, i.e.,   any graph $G$ with $\chi(G) \le k$ and $|V(G)| \le   2k+2$ 
is $\lambda$-choosable.  

If $|V(G)| \le 2k+1$, then it follows from Ohba's conjecture (which was confirmed in \cite{NRW2015}) that   $G$ is $k$-choosable, and hence $\lambda$-choosable. Assume $|V(G)|=2k+2$. As $\chi(G) \le k$ and adding edges to a non-$\lambda$-choosable graph results in a non-$\lambda$-choosable graph, it suffices to consider the case that $G$ is a complete $k$-partite graph. 
The following result was proved in \cite{zhuzhubad}.

\begin{theorem}
	\label{thm-noel}
	Assume $G$ is a complete $k$-partite graph with  $|V(G)| = 2k+2$. Then $G$ is $k$-choosable, unless $k$ is even and  $G = K_{4, 2\star (k-1)}$ or $G= K_{3\star (k/2+1), 1 \star (k/2 -1)}$.  
\end{theorem}

Thus we may assume that $k$ is even and $G = K_{4, 2\star (k-1)}$ or $G= K_{3\star (k/2+1), 1 \star (k/2 -1)}$.  
  We say a $k$-list assignment $L$ of $G$   is    {\em bad } if  $G$ is not $L$-colourable. It remains to show that any $\lambda$-list assignment of $G$ is not bad.

  All bad assignments for  $ K_{4, 2\star(k-1)}$ and $K_{3 \star (k/2+1), 1 \star (k/2-1)}$ are characterized in \cite{EOOS2002} and \cite {zhuzhubad}, respectively.
\begin{theorem}
	\label{unique3}
	Assume $G=K_{4, 2\star(k-1)}$ or  $G=K_{3 \star (k/2+1), 1 \star (k/2-1)}$, and  $L$ is a bad $k$-list assignment of $G$. Let $C= \bigcup_{v \in V(G)}L(v)$. 
	\begin{enumerate}
		\item If $G = K_{3 \star (k/2+1), 1 \star (k/2-1)}$ and the partite sets of $G$ are  $P_i=\{v_{i,1}, v_{i,2}, v_{i,3}\}$ for $1\le i\le k/2+1$ and  $P_i=\{ v_{i,1}\}$ for $k/2+2\le i\le k$,, then  $|C|=3k/2$ and for each partite set $P_i$ where $1\le i\le k/2+1$ of $K_{3\star(k/2+1), 1\star(k/2-1)}$, each colour in $C$ is contained in two of the lists $L(v_{i,1}),L(v_{i,2}),L(v_{i,3})$.
		\item  If     $G=K_{4, 2\star(k-1)}$, and the partite sets of $G$ are  $P_1=\{u_1, v_1, x_1, y_1\}$  and  $P_i=\{u_i, v_i\}$ for $2\le i\le k$,  then  
		\begin{itemize}
			\item $C$ can be partitioned into $A$ and $B$ with $|A|=|B|=k$, where $A$ can be further partitioned into $A_1,A_2,A_3,A_4$ such that $|A_1|=|A_2|$, $|A_3|=|A_4|$,  and  $B$   can be further partitioned into $B_1,B_2$ with $|B_1|=|B_2|$.
			\item $L(u_1)=A_1\cup A_3\cup  B_1$, $L(v_1)=A_1\cup A_4\cup  B_2$, $L(x_1)=A_2\cup  A_4\cup B_1$, $L(y_1)=A_2\cup A_3\cup  B_2$.
			\item $L(u_i) =A$, $L(v_i)= B$ for $2\le i  \le k$.
		\end{itemize}	
	\end{enumerate}   
\end{theorem}

Assume that
 $\lambda = \{k_1,k_2, \ldots, k_q\}$.
  Assume to the contrary that a $\lambda$-list assignment $L$ of $G$ is bad. 
  Let $C=\bigcup_{v\in V(G)}L(v)$. By definition,  $C$ can be  partitioned into the disjoint union $C_1 \cup C_2 \cup \ldots \cup C_q$ of $q$ sets, and for each $i$ and each vertex $v$ of $G$, $|L(v) \cap C_i| = k_i$.  
  
If $G=K_{3\star(k/2+1), 1\star(k/2-1)}$, then by Theorem \ref{unique3}, each color $c\in C$ is contained in two of the lists $L(v_{1,1}),L(v_{1,2}),L(v_{1,3})$, where $P=\{v_{1,1},v_{1,2},v_{1,3}\}$ is a partite set of $G$ of size $3$. 
This implies that $|C_i|=3k_i/2$, and hence $k_i$ is even for all $i$, contrary to the assumption that $m_{\lambda}({\rm odd}) > 0$.

Assume $G=K_{4,2\star(k-1)}$. Let $A_1,A_2,A_3,A_4, B_1,B_2$ and $A,B$ be the sets of colours as in Theorem \ref{unique3}.  By the definition of $\lambda$-list assignment, we know that $|A\cap C_i|=|L(u_j) \cap C_i| = |L(v_j) \cap C_i| = |B\cap C_i|=k_i$ for each $1\le i\le q$.    Assume that $$|A_1\cap C_i|=a_{1i}, \ |A_2\cap C_i|=a_{2i}, \ |A_3\cap C_i|=a_{3i}, \ |A_4\cap C_i|=a_{4i},$$ for $1 \le i\le q$, and 
$$|B_1\cap C_i|=b_{1i}, \  |B_2\cap C_i|=b_{2i}.$$ By Theorem \ref{unique3},  

\begin{equation*}
\begin{cases}
a_{1i}+a_{3i}+b_{1i}=k_i,\\
a_{1i}+a_{4i}+b_{2i}=k_i,\\
a_{2i}+a_{4i}+b_{1i}=k_i,\\
a_{2i}+a_{3i}+b_{2i}=k_i,\\
\end{cases}
\end{equation*}

It follows that $a_{1i}+a_{3i}=a_{2i}+a_{4i}$ and $a_{1i}+a_{4i}=a_{2i}+a_{3i}$. This implies that
$$a_{1i}+a_{2i}+a_{3i}+a_{4i}=|A\cap C_i|=k_i\equiv 0({\rm mod} 2),$$
  contrary to the assumption that $m_{\lambda}({\rm odd}) > 0$.
 This completes the proof for the case $m_{\lambda}(1)=0$.

\section{The  $m_{\lambda}(1)\ge 1$ case} 
In the remainder of the paper, 
assume that $k_{\lambda}=k$, $  m_{\lambda}(1)= a \ge 1$ and $m_{\lambda}({\rm odd}) -m_{\lambda}(1) = c  \ge  1$.  We need to show that 
$\phi(\lambda)=
\min\{2k  + a+c +2, 2k+3a+3 \}$. First, we   prove the upper bound, i.e., $$\phi(\lambda) \le
\min\{2k  + a+c +2, 2k +3a+3 \}.$$
By Theorem \ref{thm-previous}, $\phi(\lambda) \le
2k +a+c+2$. It remains to show that $\phi(\lambda) \le
2k +3a+3$.  Observe that $k_{\lambda}=k$, $  m_{\lambda}(1)= a$ and $m_{\lambda}({\rm odd}) =a+c$ implies that $\{1 \star a,2 \star (k-a-3c)/2,3 \star c\}$  is a refinement of $\lambda$. Hence it suffices to prove the following lemma.

\begin{lemma}
	\label{upper lemma 1}
	Assume $\lambda=\{1 \star a,2 \star b,3 \star c\}$ and $k=a+2b+3c$ (and hence $m_{\lambda}(1)=a$, $m_{\lambda}({\rm odd})=a+c$ and $k_{\lambda}=k$). Then there exists a $k$-chromatic graph $G$ with $|V(G)|=2k+3a+3$ which is not $\lambda$-choosable.
\end{lemma}

\begin{proof}
	Let $G=K_{5\star (a+1),2\star(k-a-1)}$ be the complete $k$-partite graph with $a+1$ partite sets of size 5, and $k-a-1$ partite sets of size 2. 
	\begin{itemize}
		\item The $a+1$ partite sets of size 5 are $U_i=\{u_{i,1}, u_{i,2}, u_{i,3}, u_{i,4}, u_{i,5}\}$ for $i=1, 2, \ldots, a+1$.
		\item The $k-a-1$ partite sets of size 2 are $V_j=\{v_{j,1}, v_{j,2}\}$ for $j=1, 2, \ldots, k-a-1$.
	\end{itemize} 
	Then  $G$ is a $k$-chromatic graph with $|V(G)|=2k+3a+3$.
	We will show that $G$ is not $\lambda$-choosable.

	\begin{itemize}
		\item  For $i=1,2,\ldots, c$, let $S_i=\{s_{i,1}, s_{i,2}, \ldots, s_{i,6}\}$ be pairwise disjoint sets of size $6$.
		\item  For $i=1,2,\ldots, b$,  let $T_{i}=\{t_{i,1}, t_{i,2}, t_{i,3}, t_{i,4}\}$ be pairwise disjoint sets of size $4$.
		\item 	Let $E$ be a set of $a$ colours, and the sets $E, S_i, T_{i}$ are pairwise disjoint.
	\end{itemize}
	
  Let
	\begin{eqnarray*}
&&	A_1=\bigcup_{i=1}^{c} \{s_{i,1}, s_{i,3}, s_{i,5}\} , \ A_2= \bigcup_{i=1}^{c} \{s_{i,1}, s_{i,3}, s_{i,6}\}, \ A_3= \bigcup_{i=1}^{c} \{s_{i,1}, s_{i,2}, s_{i,4}\} , \ A_4= \bigcup_{i=1}^{c} \{s_{i,2}, s_{i,3}, s_{i,4}\}, \\
&&	A_5= \bigcup_{i=1}^{c} \{s_{i,2}, s_{i,5}, s_{i,6}\} ,
	\ A_6= \bigcup_{i=1}^{c} \{s_{i,1}, s_{i,2}, s_{i,3}\} , \ A_7= \bigcup_{i=1}^{c} \{s_{i,4}, s_{i,5}, s_{i,6}\},\\
&&	B_1=\bigcup_{i=1}^{b} \{t_{i,2}, t_{i,3} \} , \ B_2= \bigcup_{i=1}^{b} \{t_{i,2}, t_{i,4}\}, \ B_3= \bigcup_{i=1}^{b} \{t_{i,1}, t_{i,2}\} , \ B_4= \bigcup_{i=1}^{b} \{t_{i,1}, t_{i,3}\}, \\
&&	B_5=\bigcup_{i=1}^{b} \{t_{i,1}, t_{i,4} \} , \ B_6= \bigcup_{i=1}^{b} \{t_{i,1}, t_{i,2}\}, \ B_7= \bigcup_{i=1}^{b} \{t_{i,3}, t_{i,4}\}.
	\end{eqnarray*}

	Let $L$ be the $\lambda$-list assignment of $G$ defined as follows:
	\[
	L(v)=\begin{cases} A_j\cup B_j\cup E, &\text{ if $v = u_{i,j}, 1\le i \le a+1, 1 \le j \le 5$}, \cr 
	A_{j+5} \cup B_{j+5} \cup E, &\text{ if $v = v_{i,j}, 1\le i \le k-a-1, 1 \le j \le 2$}, \cr 
	\end{cases}
	\]

	The colour set $C= \bigcup_{v \in V(G)}L(v)$ is partitioned into subsets $$\bigcup_{c \in E} \{c\} \cup \bigcup_{i=1}^bT_i \cup \bigcup_{i=1}^c S_i.$$
	For each vertex $v$ of $G$,  
	\begin{itemize}
		\item $|L(v) \cap \{c\}|=1$ for   $c \in E$,
		\item $|L(v) \cap T_i| =2$ for   $i=1,2,\ldots, b$,
		\item $|L(v) \cap S_i| =3$ for   $i=1,2,\ldots, c$.
	\end{itemize}
	So $L$ is a $\lambda$-list assignment.
	
	Now we show that $G$ is not $L$-colourable.
	Assume to the contrary that   $G$ has a proper $L$-colouring $\varphi$.

	For $j=1,2$, let $X_j = \{v_{1,j}, v_{2,j}, \ldots, v_{k -a-1, j}\}$. Each of $X_1, X_2$ induces a complete graph of order $k -a-1$. Hence
	$|\varphi(X_j)| = k -a-1$. Note that
	$$\varphi(X_1) \subseteq   A_6\cup B_6 \cup E, \ \varphi(X_2) \subseteq    A_7 \cup B_7 \cup E.$$
	
	For each partite set $P$ of $G$, $\bigcap_{v \in P} L(v)=E$. Hence either $\varphi(P)=\{c\}$ for some colour $c \in E$, or $|\varphi(P)| \ge 2$. Therefore 
	$$|\varphi(G)|\ge a+2(k-a)=2k-a.$$
	
	As $|\bigcup_{v \in V(G)}L(v)|=a+4b+6c=2k-a$, this implies that
	for each partite set $P$ of $G$,  
	
	\begin{equation*}
	|\varphi(P)|= 
	\begin{cases}
	1&\mbox{if $\varphi(P)=c\in E$ for some $c\in E$,}\\
	2&\mbox{otherwise.}
	\end{cases}
	\end{equation*}
	Note that if $|\varphi(P)|=2$, then $\varphi(P) \cap E = \emptyset$.

	\begin{claim}
		\label{clm-1a}
		If $U_i$ is a partite set of size $5$ and $|\varphi(U_i)|=2$, then $\varphi(U_i)\subseteq  A_6\cup B_6$.
	\end{claim}
	\begin{proof}
		Assume $1\le i\le a+1$ and $|\varphi(U_i)|=2$, say 
		$\varphi(U_i) = \{c_1,c_2\}$. Then for any $1 \le j \le 5$, $L(u_{i,j}) \cap \{c_1,c_2\} \ne \emptyset$. 
		
Assume first that $\{c_1,c_2\} \subseteq \bigcup_{i=1}^7B_i$. If $  t_{i,1}   \not\in \{c_1,c_2\}$ for any $1 \le i \le b$, then $L(u_{i,j}) \cap \{c_1,c_2\} = \emptyset$ for some $j \in \{3,4,5\}$, a contradiction.  If $  t_{i,2}   \not\in \{c_1,c_2\}$ for any $1 \le i \le b$, then $L(u_{i,j}) \cap \{c_1,c_2\} = \emptyset$ for some $j \in \{1,2,3\}$, a contradiction.
Therefore $\{c_1,c_2\}$ contains $  t_{i,1}  $ for some $1 \le i \le b$ and $  t_{i',2}  $ for some $1 \le i' \le b$.   So $\{c_1,c_2\} \subseteq B_6$.

		Assume that $\{c_1,c_2\}\nsubseteq \bigcup_{i=1}^7B_i$, say $c_1 \in \bigcup_{i=1}^7A_i$.  Let $J(c_1)=\{j: c_1 \in  L(u_{i,j})\}$. Then 
		$c_2 \in \bigcap_{j \in [5] - J(c_1)}L(u_{i,j})-E$. 
		
		If $c_1=s_{j,4}$ for some $j$, then $J(c_1)=\{3,4\}$ and 
		$L(u_{i,1}) \cap L(u_{i,2})\cap L(u_{i,5}) - E =\emptyset$, a contradiction.
		
			If $c_1=s_{j,5}$ for some $j$, then $J(c_1)=\{1,5\}$ and 
			$L(u_{i,2}) \cap L(u_{i,3})\cap L(u_{i,4}) - E =\emptyset$, a contradiction.
		
			If $c_1=s_{j,6}$ for some $j$, then $J(c_1)=\{2,5\}$ and 
			$L(u_{i,1}) \cap L(u_{i,3})\cap L(u_{i,4}) - E =\emptyset$, a contradiction.
		
		So $c_1 \ne s_{j,4}, s_{j,5}, s_{j,6}$ for   $1 \le j \le c$.

		If $c_1=s_{j,1}$ for some $j$, then $J(c_1)=\{1,2,3\}$,  and
		$L(u_{i,4}) \cap L(u_{i,5})  - E =\bigcup_{i=1}^c\{s_{i,2}\}\cup \bigcup_{i=1}^b\{t_{i,1}\} $. Hence    $\{c_1,c_2\} \subseteq  \bigcup_{j=1}^c\{s_{j,1}, s_{j,2}\}\cup \bigcup_{i=1}^b\{t_{i,1}\}$.

		If $c_1=s_{j,2}$ for some $j$, then $J(c_1)=\{3,4,5\}$, and
		$L(u_{i,1}) \cap L(u_{i,2}) - E =\bigcup_{i=1}^c\{s_{i,1}, s_{i,3}\}\cup \bigcup_{i=1}^b\{t_{i,2}\}$. Hence  $\{c_1,c_2\} \subseteq  \bigcup_{j=1}^c\{s_{j,1}, s_{j,2}, s_{j,3}\}\cup \bigcup_{i=1}^b\{t_{i,2}\}$.

		If $c_1=s_{j,3}$ for some $j$, then $J(c_1)=\{1,2,4\}$, and
		$L(u_{i,3}) \cap L(u_{i,5}) - E =\bigcup_{i=1}^c\{s_{i,2}\}\cup \bigcup_{i=1}^b\{t_{i,1}\}$. Hence  $\{c_1,c_2\} \subseteq  \bigcup_{j=1}^c\{  s_{j,2}, s_{j,3}\}\cup \bigcup_{i=1}^b\{t_{i,1}\}$. 
		
		In any case, $\varphi(U_i)=\{c_1,c_2\}\subseteq \bigcup_{j=1}^c\{s_{j,1},s_{j,2},s_{j,3}\}\cup  \bigcup_{i=1}^b\{t_{i,1}, t_{i,2}\} = A_6\cup B_6$.
		This completes the proof of the claim.
	\end{proof}

	It follows from Claim \ref{clm-1a} that  $$\varphi(\bigcup_{i=1}^{a+1}U_i)\subseteq E\cup A_6\cup B_6.$$
	Hence $$\varphi(X_1) \cup \varphi(\bigcup_{i=1}^{a+1}U_i)\subseteq E\cup A_6\cup B_6.$$
 As $|E|=a$, $|\varphi(U_i)|=2$ for some $i$. This implies that $$k+1 \le |\varphi(X_1)|+| \varphi(\bigcup_{i=1}^{a+1}U_i)| \le |E \cup A_6 \cup B_6| =k, $$
   a contradiction. 
\end{proof}



Next we prove the lower bound that $\phi(\lambda ) \geqslant \min\{2k+3a+3, 2k+a+c+2\}$.
 
Assume to the contrary that  $\phi(\lambda) < \min\{2k   +a+c +2, 2k +3a+3 \}$ for some $\lambda$.
We choose such a multi-set  $\lambda=\{k_1,k_2,\ldots, k_q\}$  with $|\lambda|=q$   minimum. Assume that  $k_1=k_2=\ldots=k_a=1$ and $3\le k_{a+1}\le k_{a+2}\le \ldots \le k_{a+c}$ are the odd integers in $\lambda$.

Let  $n=\min\{2k  +a+ c +2, 2k  + 3a+3 \}$.  Then there is a $k $-chromatic graph  $G$ with 
$|V(G)| \le n-1$ which is not $\lambda$-choosable.
We may assume that $G$ is a complete $k $-partite graph with $|V(G)|=n-1$ and with partite sets $P_1,P_2,\ldots, P_{k }$ such that $|P_1| \ge |P_2| \ge \ldots \ge |P_{k }|$. 
 For a positive integer $i$, let 
$$I_i=\{j: |P_j|=i\}.$$
Note that $|P_1| \ge 3$   (as $|V(G)| > 2k$).
If $c\le 2a+1$, then $n=2k+a+c+2$ and if $c\ge 2a+2$, then $n=2k+3a+3$.

\begin{claim}
	\label{clm-p_1} 
	$|P_1| \le 4$, and if $c\le 2a+1$, then $|P_1| \le   c-2a+3 $.
\end{claim}
\begin{proof}		
	Let   $G'=G-P_1$ and $L'(v)=L(v)-C_1$ for $v \in V(G')$. Then  $L'$ is a $\lambda'$-list assignment of $G'$, where $\lambda'=\lambda-\{1\}$. 
	Note that   $m_{\lambda'}(1)= a'=a-1$,    $m_{\lambda'}({\rm odd}) -m_{\lambda'}(1) = c' =c$ and $k'=k_{\lambda'}=k-1$.

	If $|P_1| \ge 5$, then $|V(G')| \le |V(G)|-5 =\min\{2k+a+c+1-5,2k+3a+2-5\} =\min\{2k'+a'+c'-1, 2k'+3a'+2\}$. Hence $G'$ is $\lambda'$-choosable and has a proper  $L'$-colouring of $G'$, which extends to a proper $L$-colouring by colouring the vertices in $P_1$ by colours from $C_1$, a contradiction. 
	So $|P_1| \le 4$.

	If  $c\le 2a+1$ and $|P_1| \ge   c-2a+4$, then    $|V(G')|= 2k+a+c+1-|P_1| \le  \min\{ 2k'+a'+c'+1,  2k'+3a'+2\}$ (as $|P_1| \ge 3$). Hence $G'$ is $\lambda'$-choosable, a  contradiction.  
\end{proof}

 Since $a \ge 1$, it follows from Claim \ref{clm-p_1} that {$c\ge 2a\ge 2$}, and if $c=2$, then $a=1$ and $|P_1|=3$. 

\begin{definition}
	\label{def-badtuple}
	A 4-tuple $(a_1,a_2,a_3,a_4)$ of integers is {\em reducible} if $$0 \le a_i \le |I_i|,  \  \sum_{i=1}^4a_i=k_{a+1} \text{ and } 2k_{a+1}+1 \le \sum_{i=1}^4ia_i \le 2k_{a+1}+2.$$
\end{definition}

\begin{claim}
	\label{clm-k_1}
	There is no reducible 4-tuple.
\end{claim}

\begin{proof}
	Assume to the contrary that $(a_1,a_2,a_3,a_4)$ is a reducible 4-tuple. 
	Let $X$ be a subgraph of $G$ consisting of $a_i$ partite sets of $G$ of size $i$ for $1 \le i \le 4$, and let $G_1=G[X]$. Then $G_1$ is a complete $k_{a+1}$-partite graph with 
	$  |V(G_1)|\le 2k_{a+1}+2$. By Theorem \ref{thm-noel}, $G_1$ is $k_{a+1}$-choosable, and hence  there is a proper $L$-colouring $\psi$ of $G_1$ using colours from $C_{a+1}$. 
	
	Let $\lambda'=\lambda-\{k_{a+1}\}$, $G'=G-X$, and $L'(v) = L(v)-C_{a+1}$  for each vertex $v$ of $G'$. Then $L'$ is a $\lambda'$-list assignment of $G'$.
	Note that $k'=k_{\lambda'}=k-k_{a+1}$, $a'=m_{\lambda'}(1)=a$ and $c'= m_{\lambda'}({\rm odd}) - m_{\lambda'}(1)=c-1$, and $G'$ is a complete $k'$-partite graph with $|V(G')| \le |V(G)|- (2k_{a+1}+1)$.

	If $c\le 2a+1$, then $c'=c-1\le 2a=2a'$ and 
	\begin{eqnarray*}
		|V(G')|&\le& 2k+a+c+1-(2k_{a+1}+1)=2k'+a'+c'+1.
	\end{eqnarray*}

	If $c\ge 2a+2$, then $c'=c-1\ge 2a+1=2a'+1$ and	
	\begin{eqnarray*}
		|V(G')|&\le& 2k+3a+2-2k_{a+1}-1=2k'+3a'+1\le 2k'+a'+c'.
	\end{eqnarray*}	
	
	By the minimality of $|\lambda|$, $G'$ is $\lambda'$-choosable. Let $\phi$ be an $L'$-colouring of $G'$. Then $\phi \cup \psi$ is a proper $L$-colouring of $G$,  
	a contradiction.
\end{proof}

If $|I_2| \ge k_{a+1}-1$, then since $|I_3|+|I_4| \ge 1$, $(0,k_{a+1}-1, 0,1 )$ or $(0,k_{a+1}-1, 1,0 )$ is a reducible 4-tuple, contrary to Claim \ref{clm-k_1}. So 
\begin{equation}
\label{eq-1}
 |I_2| \le k_{a+1}-2.
\end{equation}

\begin{claim}
	\label{clm-I_1}
 If $c \ge 3$, then $|I_1| \ge  \frac 23 k_{a+1}$.  If $c=2$, then $ |I_1| \ge (k_{a+1}-1)/2$.
\end{claim}
\begin{proof}
	Observe that $3(k-|I_2|-|I_1|)+2|I_2|+|I_1|\le |V(G)|= n-1$ and $k \ge a+ck_{a+1}$. Therefore $|I_2|+2|I_1|\ge 3k-n+1$. If $c\le 2a+1$, then $n=2k+a+c+2$ and hence 
	 $$
	 2|I_1|\ge k-a-c-1-|I_2| \ge (a+c k_{a+1})-a-c-1-(k_{a+1}-2)= (c-1)(k_{a+1}-1). 
	$$ 
	 If $c\ge 2a+2$, then $n=2k+3a+3$ and hence 
 $$
	 2|I_1|\ge k-3a-2-|I_2| \ge   (c-1)k_{a+1}-2a\ge (c-1)(k_{a+1}-1).
$$ 	

If $c \ge 3$, then $|I_1| \ge k_{a+1}-1 \ge \frac 23 k_{a+1}$  as $k_{a+1}\ge 3$. 

If $c=2$, then $a=1$ and $|P_1|=3$ and $n=2k+a+c+2=2k+5$. Hence $$3|I_3|+2|I_2|+|I_1|=n-1=2k+4 \text{  and } |I_3|+|I_2|+|I_1|=k.$$ Therefore $2|I_3|+|I_2|=k+4$.  As $|I_2|+2|I_1|\ge k-a-c-1=k-4$ and hence $$2k=2(|I_1|+|I_2|+|I_3|)=(2|I_3|+|I_2|)+(|I_2|+2|I_1) \ge (k+4)+(k-4)=2k. $$ Therefore $k-4=|I_2|+2|I_1| \le k_{a+1}-2 + 2|I_1| $. As $k\ge 1+2 k_{a+1} $, we have $ |I_1| \ge (k_{a+1}-1)/2$. 
\end{proof}

\medskip
\noindent	{\bf Case 1}:   $|P_1|=4$.
\medskip

Let $b_1= \frac{ k_{a+1}-|I_2|-1}{2}$. Note that $\lceil b_1 \rceil+|I_2|+ \lfloor b_1 \rfloor+1=k_{a+1}$ and  $\lceil b_1 \rceil+2|I_2|+ 3\lfloor b_1 \rfloor+4\times 1=2k_{a+1}+2$ (or $2k_{a+1}+1$) if $2b_1$ is even (or odd). As $c \ge 3$, by  Calim \ref{clm-I_1},    $|I_1|\ge \frac{2}{3}k_{a+1}\ge \lceil b_1 \rceil$. 

If $|I_3|\ge \lfloor b_1 \rfloor$, then $(\lceil b_1 \rceil, |I_2|, \lfloor b_1 \rfloor, 1 )$ is a reducible 4-tuple,  contrary to Claim \ref{clm-k_1}. Therefore,

\begin{equation}
\label{eq-x_3}
|I_3|< \lfloor \frac{k_{a+1}-|I_2|-1}{2} \rfloor.
\end{equation}

Let $b_2=\frac{k_{a+1}-|I_2|-2|I_3|-1}{3}$. Note that $|I_1|\ge\frac{2}{3}k_{a+1} \ge  2 \lceil b_2\rceil +|I_3|$.

If $|I_4|\ge \lceil b_2 \rceil +1$, then it is straightforward to check that  
	
\begin{itemize}
	\item either $3b_2\equiv 0 \pmod{3}$, and $(2b_2+|I_3|,|I_2|,|I_3|,b_2+1)$ is  a reducible 4-tuple,  
	\item or $3b_2\equiv 1 \pmod{3}$, and $(2\lceil b_2\rceil +|I_3|-1,|I_2|,|I_3|,\lfloor b_2\rfloor +1)$ is  a reducible 4-tuple, 
	\item or $3b_2\equiv 2 \pmod{3}$,  and  $(2\lceil b_2\rceil +|I_3|,|I_2|-1,|I_3|,\lceil b_2\rceil +1)$ is  a reducible 4-tuple,
\end{itemize}
a contradiction.
Therefore,
\begin{equation}
\label{eq-x_4}
|I_4|<  \lceil b_2 \rceil +1 = \lceil\frac{k_{a+1}-|I_2|-2|I_3|-1}{3} \rceil +1.
\end{equation}

It follows from  (\ref{eq-1}), (\ref{eq-x_3}) and (\ref{eq-x_4}) that \begin{eqnarray*}
	|V(G)|&=&4|I_4|+3|I_3|+2|I_2|+|I_1|\\
	&\le& 4\times \frac{k_{a+1}-|I_2|-2|I_3|-1}{3}+4+3|I_3|+2|I_2|+|I_1|\\
	&\le&\frac{4}{3}k_{a+1}+\frac{2}{3}|I_2|+\frac{1}{3} \frac{k_{a+1}-|I_2|-1}{2}+|I_1|+\frac{8}{3}\\
	&\le&\frac{3}{2}k_{a+1}+\frac{1}{2}(k_{a+1}-2)+|I_1|+\frac{5}{2}\\
	&\le &2k+\frac{1}{2}.
\end{eqnarray*}
The last inequality holds because $k\ge \sum_{i=1}^{c}k_{a+i}\ge ck_{a+1}\ge 2k_{a+1}$ and $|I_1|\le k-1$. This contradicts to $|V(G)|=n-1\ge 2k+1$.

\medskip
\noindent	{\bf Case 2}:   $|P_1|=3$.
\medskip

Let $b_3=\frac{ k_{a+1}-|I_2|-1}{2}$. Note that $\lfloor b_3 \rfloor+|I_2|+ (\lceil b_3 \rceil+1)=k_{a+1}$ and  $\lfloor b_3 \rfloor+2|I_2|+ 3(\lceil b_3 \rceil+1)=2k_{a+1}+2$ (or $2k_{a+1}+1$) if $2b_3$ is odd (or even).  

By  Claim \ref{clm-I_1}, $|I_1|\ge \frac{1}{2}(k_{a+1}-1)\ge b_3$. 

If $|I_3|\ge \lceil b_3 \rceil+1$, then $(\lfloor b_3 \rfloor, |I_2|, \lceil b_3 \rceil+1, 0 )$ is a reducible 4-tuple,  contrary to Claim \ref{clm-k_1}. Therefore,

\begin{equation}
\label{eq-x3 x=3}
|I_3|< \lceil \frac{k_{a+1}-|I_2|-1}{2} \rceil+1.
\end{equation}

It follows from (\ref{eq-1}) and (\ref{eq-x3 x=3}) that
 \begin{eqnarray*}
	|V(G)|
	&\le& 3\times \frac{k_{a+1}-|I_2|-1}{2}+3+2|I_2|+|I_1|\\
	&\le&\frac{3}{2}k_{a+1}+\frac{1}{2}(k_{a+1}-2)+|I_1|+\frac{3}{2}\\
	&\le&2k-\frac{1}{2}.
\end{eqnarray*}
The last inequality holds because $k\ge \sum_{i=1}^{c}k_{a+i}\ge ck_{a+1}\ge 2k_{a+1}$ and $|I_1|\le k-1$. This contradicts to $|V(G)|=n-1\ge 2k+1$.

This completes the proof of Theorem \ref{thm-main}.


	\bibliography{Reference}

\end{document}